\documentclass[a4paper]{amsart}

\usepackage{tikz}
\usepackage{pgflibraryarrows}
\usepackage{pgflibrarysnakes}

\usepackage{amsfonts}
\usepackage{amssymb}
\usepackage{amsthm}
\usepackage{amsmath}
\usepackage[all]{xy}

\SelectTips{cm}{}

\theoremstyle{plain}
\newtheorem{thm}{Theorem}[section]
\newtheorem{prop}[thm]{Proposition}
\newtheorem{cor}[thm]{Corollary}
\newtheorem{lem}[thm]{Lemma}

\theoremstyle{definition}
\newtheorem{defn}[thm]{Definition}

\theoremstyle{remark}

\numberwithin{equation}{section}

\newcommand{\sL}{\mathcal{L}}
\newcommand{\sN}{\mathcal{N}}
\newcommand{\sS}{\mathcal{S}}

\newcommand{\sStw}{\widetilde{\sS}}

\newcommand{\bL}{\mathbf{L}}

\newcommand{\bt}{\mathbf{t}}
\newcommand{\bs}{\mathbf{s}}

\newcommand{\br}{\mathbf{r}}
\newcommand{\bbr}{\bar{\mathbf{r}}}

\newcommand{\wL}{\widetilde{L}}
\newcommand{\wsN}{\widetilde{\sN}}
\newcommand{\wrho}{\widetilde{\rho}}

\newcommand{\del}{\partial}

\newcommand{\Gh}{\widehat{G}}

\newcommand{\CC}{\mathbb{C}}

\newcommand{\NN}{\mathbb{N}}
\newcommand{\RR}{\mathbb{R}}
\newcommand{\QQ}{\mathbb{Q}}

\newcommand{\ZZ}{\mathbb{Z}}

\newcommand{\id}{\textup{id}}

\newcommand{\nin}{\noindent}

\newcommand{\ra}{\rightarrow}
\newcommand{\xra}{\xrightarrow}
\newcommand{\lra}{\longrightarrow}
\newcommand{\co}{\colon\!}

\newcommand{\BO}{\textup{BO}}
\newcommand{\BTOP}{\textup{BTOP}}
\newcommand{\G}{\textup{G}}
\newcommand{\TOP}{\textup{TOP}}
\newcommand{\Gsign}{\textup{G-sign}}
\renewcommand{\mod}{\textup{mod }}
\newcommand{\red}{\textup{red}}
\newcommand{\res}{\textup{res}}
\renewcommand{\max}{\textup{max}}

\newcommand{\ol}{\overline}
\newcommand{\sign}{\textup{sign}}
\newcommand{\glue}{\textup{glue}}
\newcommand{\glueh}{h'}

\newcommand{\RhG}{R_{\widehat G}}

\newcommand{\RhGm}{R_{\widehat G}^-}

\title[Higher simple structure sets of lens spaces]{Higher simple structure sets of lens spaces \\ with the fundamental group of order $2^K$}
\author{L'udov\'it Balko}
\author{Tibor Macko}
\author{Martin Niepel}
\author{Tom\'a\v s Rusin}



\address{Faculty of Mathematics, Physics, and Informatics, Comenius University, Mlynsk\'a dolina,
SK-824 48, Bratislava, Slovakia} \email{ludovit.balko@fmph.uniba.sk} \email{martin.niepel@fmph.uniba.sk} \email{tomas.rusin@fmph.uniba.sk}

\address{Institute of Mathematics, Slovak Academy of Sciences, \v Stef\'anikova 49, Bratislava, SK-81473, Slovakia} \email{macko@mat.savba.sk}

\thanks{This work was supported by the grant VEGA 1/0101/17, and by the Slovak Research and Development Agency under the contract No. APVV-16-0053.}

\date{\today}

\begin{document}

\maketitle

\begin{abstract}
	Extending work of many authors we calculate the higher simple structure sets of lens spaces in the sense of surgery theory in the case when the fundamental group has order a power of $2$. As a corollary we also obtain a calculation of the simple structure set of the products of lens spaces and spheres of dimension grater or equal to $3$ in this case.
\end{abstract}



\section*{Introduction}
\label{sec:intro}


The main result of this paper, Theorem~\ref{thm:main-thm}, calculates the topological higher simple structure sets $\sS^{s}_{\del} (L \times D^{m})$ in the sense of surgery theory, see Section~\ref{sec:higher-str-sets} for a definition, for $L$ a fake lens space with the fundamental group $\ZZ/N \cong G = \pi_{1} (L)$ for $N = 2^{K}$ and $m \geq 1$. The case $m=0$ was done by Wall for $N$ odd \cite[14.E]{Wall(1999)}, by Lopez de Medrano \cite{LdM(1971)} and Wall for $N=2$ \cite[14.D]{Wall(1999)} and by Macko and Wegner for $N=2^{K}$ \cite{Macko-Wegner(2009)} and for general $N = 2^{K} \cdot M$, where $K \geq 1$ and $M$ is odd \cite{Macko-Wegner(2011)}. The case $m \geq 1$ was done by Madsen and Rothenberg for $N=2$ and $N$ odd \cite{Madsen-Rothenberg(1989)}. Hence the remaining cases are when $m \geq 1$ and $N = 2^{K} \cdot M$, where $K \geq 1$ and $M$ is odd. In this paper we take care of the case $N=2^K$ and we plan to address the most general case in a subsequent work. We also obtain Corollary~\ref{cor:higehr-str-sets-L-times-S} where we calculate $\sS^{s} (L \times S^m)$ for $m \geq 3$.


The calculations build on the work from all the above mentioned sources. The basic idea is to extend the methods of \cite{Macko-Wegner(2009)}, which in turn use \cite{Wall(1999)}, from the case $m=0$ to the case $m \geq 1$. Several issues need to be verified and adjusted in the present case, which is what we do. We use the $\rho$-invariant for manifolds with boundary from \cite{Madsen-Rothenberg(1989)} and some of its general properties, but we need to do a bit more in our special case $N=2^K$. Most importantly, a formula of Wall for the $\rho$-invariant of certain closed manifolds from \cite[Theorem 14C.4]{Wall(1999)} needs to be generalized and slightly modified. The generalization and its proof is the main technical result of this paper. It is formulated in Theorem~\ref{thm:rho-formula-cp-times-disk} and the proof is based on a technical Proposition~\ref{prop:alpha-and-beta-depend-linearly-on-s-4i}. 

The paper is organized as follows. In Section~\ref{sec:results} we state the main Theorem~\ref{thm:main-thm} and Corollary~\ref{cor:higehr-str-sets-L-times-S}. Section~\ref{sec:higher-str-sets} contains the definition of the higher simple structure sets for a general manifold $X$ and their properties. Section~\ref{sec:higher-str-set-for-cp} deals with the cases when $X$ are complex projective spaces, which we also need as explained at the beginning of that section and Proposition~\ref{prop:alpha-and-beta-depend-linearly-on-s-4i} is proved. Section~\ref{sec:lens-times-disk} contains the first results about the cases when $X$ are fake lens spaces based on the long surgery exact sequence, which reduces the calculation to a certain extension problem, see Theorem~\ref{thm:how-to-split-alpha-k}. In Section~\ref{sec:the-rho-invariant} the $\rho$-invariant is discussed which is the main tool to solve the remaining extension problem, in particular the above mentioned Theorem~\ref{thm:rho-formula-cp-times-disk} is proved. The formula from this theorem is then used in Section~\ref{sec:calculations} which contains the proof of the main Theorem~\ref{thm:main-thm} and Corollary~\ref{cor:higehr-str-sets-L-times-S}. The final Section~\ref{sec:final-remarks} contains a discussion of further directions.


\section{Results}
\label{sec:results}


A {\it fake lens space} $L = L (\alpha)$ is a topological manifold given as the orbit space of a free action $\alpha$ of a finite cyclic group $G = \ZZ/N$ on a sphere $S^{2d-1}$. The main result of this paper is the following theorem about them.

\begin{thm}\label{thm:main-thm}
	Let $L = L (\alpha)$ be a $(2d-1)$-dimensional fake lens space for some free action $\alpha$ of the cyclic group $G = \ZZ/N$ with $N = 2^K$ for some $K \geq 1$ on $S^{2d-1}$ with $d \geq 2$ and let $k \geq 1$. Then we have isomorphisms
	\begin{align*}
		(\wrho_{\del},\bbr_{0},\bbr,\br) \co \sS^{s}_{\del} (L \times D^{2k}) & \xra{\cong} 
		\begin{cases}
			F^{+} \oplus \ZZ \oplus T'_{N} \oplus T_{2} \quad & d = 2e, k=2l \\
			F^{-} \oplus \ZZ/2 \oplus T'_{N} \oplus T_{2} \quad & d = 2e, k=2l+1 \\
			F^{-} \oplus \ZZ \oplus T'_{N} \oplus T_{2} \quad & d = 2e+1, k=2l \\
			F^{+} \oplus \ZZ/2 \oplus T'_{N}  \oplus T_{2} \quad & d = 2e+1, k=2l+1 
		\end{cases} \\
		(\bbr,\br) \co \sS^{s}_{\del} (L \times D^{2k+1}) & \xra{\cong} T_{2} (\textup{odd}) \quad \textup{also} \; k=0,
	\end{align*}
	where the meaning of the symbols in the target is as follows.
	\begin{enumerate}
		\item $F^{+}$ is a free abelian group of rank $2^{K-1}$;
		\item $F^{-}$ is a free abelian group of rank $2^{K-1} - 1$;
		\item Let $c_N (d,k) = e-1$ when $(d,k)=(2e,2l)$ and let $c_N (d,k) = e$ in other cases. Then 
		 \[ 
		 T'_{N} \cong \bigoplus_{i=1}^{c_N(d,k)} \ZZ/2^{\textup{min} \{ 2i , K \}}; 
		 \]
		\item Let  $c_2 (d,k) = e$ when $(d,k)=(2e+1,2l)$ and let $c_2 (d,k) = e-1$ in other cases. Then 
		 \[ 
		 T_{2} \cong \bigoplus_{i=1}^{c_2 (d,k)} \ZZ/2; 
		 \]  
		\item Let  $c_2 (d,k,\textup{odd}) = e-1$ when $(d,k)=(2e,2l+1)$ and let $c_2 (d,k,\textup{odd}) = e$ in other cases. Then
		 \[ 
		 T_{2} (\textup{odd}) \cong \bigoplus_{i=1}^{c_2 (d,k,\textup{odd})} \ZZ/2; 
		 \] 
		\item The symbol $\wrho_{\del}$ denotes the reduced $\rho$-invariant for manifolds with boundary;
		\item The invariant $\bbr$ is an invariant derived from the splitting invariants along $4i$-dimensional submanifolds;
		\item The invariant $\br$ consists of the splitting invariants along $(4i-2)$-dimensional submanifolds;
		\item The invariant $\bbr_{0}$ is an invariant derived from the splitting invariants along $2k$-dimensional submanifolds.
	\end{enumerate}
\end{thm}

The invariant $\wrho_{\del}$ is defined in Section~\ref{sec:the-rho-invariant}, the invariants $\br$ are defined in Section~\ref{sec:lens-times-disk} and the invariants $\bbr_{0}$ and $\bbr$ are defined in Section~\ref{sec:calculations}.

\begin{cor} \label{cor:higehr-str-sets-L-times-S}
	For $k \geq 2$ we have isomorphisms 
	\[	
		(\red{_\del},\br',\br'') \co \sS^{s} (L \times S^{2k}) \cong \sS^{s}_{\del} (L \times D^{2k}) \oplus T_N (d) \oplus T_2 (d)
	\]
	and for $k \geq 1$ we have isomorphisms 
	\[	
		(\red{_\del},\br',\br'') \co\sS^{s} (L \times S^{2k+1}) \cong \sS^{s}_{\del} (L \times D^{2k+1}) \oplus T_N (d) \oplus T_2 (d)
	\]
	where 
	\[
	T_N (d) \cong \bigoplus_{i=1}^{\lfloor (d-1)/2 \rfloor} \ZZ/2^{K} \quad \textup{and} \quad T_2 (d) \cong \bigoplus_{i=1}^{\lfloor d/2 \rfloor} \ZZ/2.
	\]
\end{cor}
The map $\red_\del$ is explained in Section~\ref{sec:higher-str-sets} and the invariants $\br',\br''$ in Section~\ref{sec:calculations}. Together with Theorem \ref{thm:main-thm} this shows that $\sS^{s} (L \times S^{m})$ is calculated by the invariants $\wrho_{\del},\bbr_{0},\bbr,\br,\br', \br''$ where each symbol has to be appropriately interpreted depending on parity of $d$ and $m$.


\section{Higher structure sets and the long surgery exact sequence}
\label{sec:higher-str-sets}


We review basic definitions and properties of the higher simple structure sets from surgery theory which we use. More detailed and more comprehensive information can be found in \cite{Wall(1999)}, \cite{Ranicki(2002)}, \cite{Crowley-Lueck-Macko(2019)}, \cite{Kirby-Siebenmann(1977)}, \cite{Quinn(1970)}, \cite{Madsen-Rothenberg(1989)}, \cite{Weiss-Williams(2001)}.

Let~$X$ be a closed~$n$-dimensional topological manifold. The {\it simple structure set}~$\sS^{s} (X)$ of~$X$ in the sense of surgery theory is defined to be the set of equivalence classes of simple homotopy equivalences~$h \co M \ra X$, with the source an~$n$-dimensional closed manifold, modulo homeomorphism up to homotopy in the source. Knowledge of~$\sS^{s} (X)$ is generally regarded as understanding manifolds in the simple homotopy type of~$X$. Many calculations are known, e.g. for $X = S^n$, $\RR P^n$, $\CC P^n$, $T^n = (S^{1})^{\times n}$, lens spaces, see e.g. \cite{Wall(1999)}.

Let~$Y$ be a compact~$n$-dimensional manifold with (a possibly empty) boundary. Then the simple structure set $\sS^{s}_{\del} (Y)$ is defined to be the set of equivalence classes of simple homotopy equivalences~$h \co (M,\del M) \ra (Y,\del Y)$, with the source an~$n$-dimensional compact manifold with boundary and whose restriction~$h| \co \del M \ra \del Y$ is a homeomorphism, modulo homeomorphism up to homotopy relative boundary in the source. We regard knowledge of~$\sS^{s}_{\del} (Y)$ as understanding manifolds in the simple homotopy type of~$Y$ relative to~$\del Y$.

If $X$ is closed we can take for any $m \geq 1$ the compact manifold with boundary $Y = X \times D^{m}$ and consider $\sS^{s}_{\del} (X \times D^{m})$. It turns out that these simple structure sets form a group, where the group operation is obtained geometrically by ``stacking''. In fact there is a space sometimes denoted $\sStw^{s} (X)$ whose $m$-th homotopy group is~$\sS^{s}_{\del} (X \times D^{m})$, \cite{Quinn(1970)} (this includes the case $m=0$). Therefore we sometimes call the simple structure sets of $X \times D^m$ {\it higher simple structure sets} of $X$.

The space~$\sStw^{s} (X)$ is closely related to automorphism spaces of~$X$ and hence its knowledge not only tells us about the manifolds in the homotopy type of $X \times D^{m}$ relative $X \times S^{m-1}$, but it also possibly tells us something about the space of self-homeomorphisms of $X$ (see \cite{Weiss-Williams(2001)} for more details).

On the other hand, given $X$, we might be interested in the simple structure sets of closed manifolds $X \times S^{m}$ for some $m \geq 1$. If $m \geq 3$ then transversality, restriction and the $\pi-\pi$-theorem of \cite[Chapter 4]{Wall(1999)} provide us with a map denoted $\res_{\pitchfork} \co \sS^{s} (X \times S^m) \ra \sS^{s} (X \times D^{m},X \times S^{m-1})$, where now $\sS^{s} (Y,\del Y)$ is yet another version of the structure set where we do not require homeomorphism on the boundary. When $Y = X \times D^{m}$ with $m \geq 3$, the set $\sS^{s} (Y,\del Y)$ is well calculable, again by the $\pi-\pi$-theorem as we explain below, see ~\eqref{eqn:str-set-rel-not-rel-bdry}. The above map is surjective by taking a double and the kernel is $\sS^{s}_{\del} (X \times D^{m})$ so that we have
\begin{equation} \label{eqn:higher-str-set-X-times-S}
	(\red_\del,\res_{\pitchfork}) \co \sS^{s} (X \times S^m) \cong \sS^{s}_{\del} (X \times D^{m}) \times \sS^{s} (X \times D^{m},X \times S^{m-1})
\end{equation}
and hence any knowledge of the higher structure sets also tells us about~$\sS^{s} (X \times S^{m})$. 

Thanks to the $s$-cobordism theorem elements in $\sS^{s}_{\del} (X \times D^{m})$ can also be represented by simple homotopy equivalences $h \co X \times D^{m} \ra X \times D^{m}$ so that the restriction of $h$ to the product of $X$ with the lower hemisphere of the boundary $S^{m-1}$ is the identity and the restriction of $h$ to the product of $X$ with the upper hemisphere is some homeomorphism (which a-priori does not commute with the projection to that hemisphere), see \cite{Madsen-Rothenberg(1989)}. Hence the source manifold is fixed in this description which may be of advantage in some constructions. 

There are versions of the above concepts where the word simple is dropped, but we will not use them in the present paper. Of course. if the corresponding Whitehead group vanishes there is no difference, see~\cite[Chapters 1,2]{Crowley-Lueck-Macko(2019)}. Since this is the case in the simply-connected situation we often leave out the word simple when dealing with such manifolds.

The main tool for computing $\sS^{s}_{\del} (X \times D^{m})$ for a given $a$-dimensional manifold $X$ with $G = \pi_1 (X)$ and $m \geq 0$, so that the dimension of $X \times D^{m}$ is $n = a+m \geq 5$, is the long surgery exact sequence: 
\begin{equation} \label{eqn:surgery-exact-sequence}
\begin{split}
	\cdots \xra{\eta} \sN_\partial (X \times D^{m+1}) \xra{\theta}
	L^s_{n+1} (\ZZ G) & \xra{\partial} \\ \quad \quad \xra{\partial} \sS_\del^s (X \times D^{m})
	\xra{\eta} \sN_\partial (X \times D^{m}) & \xra{\theta} L^s_{n}
	(\ZZ G) \xra{\del} \cdots.
\end{split}
\end{equation}
For detailed explanations of the terms we refer the reader to \cite[Chapter 10]{Wall(1999)} or \cite[Chapter 10]{Crowley-Lueck-Macko(2019)} and \cite{Kirby-Siebenmann(1977)} in the topological case. Here we only review the facts that are essential in this paper.

The sequence \eqref{eqn:surgery-exact-sequence} is a long exact sequence of abelian groups, which is a geometric fact for the terms with $m \geq 2$ and for smaller $m$ it is shown using algebraic theory of surgery of Ranicki \cite{Ranicki(1992)}. The $L$-groups are $4$-periodic in $n$ and can be defined algebraically using quadratic forms. They are functorial in $G$ and using the functoriality it is convenient to denote 
\begin{equation} \label{eqn:reduced_L-group}
	L_{n}^{s} (\ZZ G) \cong L_{n} (\ZZ) \oplus \widetilde{L}_{n}^{s} (\ZZ G).
\end{equation}
The normal invariants \cite[Chapter 6]{Crowley-Lueck-Macko(2019)} are a generalized cohomology theory
\begin{equation} \label{eqn:normal-invariants-general-formula}
	\begin{split}
		\sN_\partial (X \times D^{m}) \cong [X \times D^{m},X \times S^{m-1} ; \G/\TOP,\ast] & \cong \\ 
		\cong H^{0} (X \times D^{m},X \times S^{m-1} ; \bL_{\bullet} \langle 1 \rangle) \cong & H_{n} (X ; \bL_{\bullet} \langle 1 \rangle)
	\end{split}
\end{equation}
whose coefficients spectrum~$\bL_{\bullet} \langle 1 \rangle$ is well understood, its associated infinite loop space is the well known space $\G/\TOP$ with homotopy groups
\begin{equation} \label{eqn:htpy-grps-g-mod-top}
	\pi_{n} (\G/\TOP) \cong L_{n} (\ZZ) \cong \ZZ, 0, \ZZ/2, 0 \quad \textup{for} \; n \equiv 0,1,2,3 \; (\mod \; 4), \; n \geq 1.
\end{equation}
Its homotopy type is recalled in \eqref{eqn:htpy-type-of-g-top}. In particular, it also possesses an almost $4$-periodicity. 

Elements in $\sN_{\del} (X \times D^{m})$ can be represented by degree one normal maps of the form $(f,\ol f) \co (M,\del M) \ra (X \times D^{m},X \times S^{m-1})$ where the restriction of $f$ to $\del M$ is a homeomorphism. The map $\theta \co \sN_\partial (X \times D^{m}) \ra L_{n} (\ZZ G)$ is called the surgery obstruction map. The summand $L_{n} (\ZZ)$ is always hit by this map thanks to the existence of the Milnor and Kervaire manifolds \cite{Madsen-Milgram(1979)}.

In fact, we will need a more detailed description of the surgery obstruction map $\theta \co \sN (X) \ra L_{n} (\ZZ)$ in the case $X$ is closed with $\pi_{1} (X) = \{ 1\}$ and with dimension $n=4i$. Let $(f,\ol f) \co M \ra X$ be a degree one normal map representing an element in $\sN (X)$ for such an $X$. These data contain in particular the bundle map $\ol f \co \tau_M \ra \xi$ from the stable tangent microbundle $\tau_M$ to some stable microbundle $\xi$ over $X$.\footnote{The microbundles are used here since we are in the topological category, see~\cite{Kirby-Siebenmann(1977)}} Then under the identification $L_{4i} (\ZZ) \cong \ZZ$ of \eqref{eqn:htpy-grps-g-mod-top} the surgery obstruction is the difference of signatures divided by $8$
\begin{equation} \label{eqn:surgery-obstruction-is-difference-of-signatures}
	\theta (f, \ol f) = 1/8 \cdot (\sign (M) - \sign (X)) = 1/8 \cdot (\ell (M)[M] - \ell (X)[X]) \in \ZZ,	
\end{equation}
where $\ell$ denotes the total Hirzebruch $\ell$-class which is a rational linear combination of the rational Pontrjagin classes of the tangent microbundle, which are well defined for topological manifolds due to the homotopy equivalence of the rationalized classifying spaces $\BO_{\QQ} \simeq \BTOP_{\QQ}$, \cite{Kirby-Siebenmann(1977)}. In general we have the class $\ell (\xi) \in H^{4 \ast} (X;\QQ)$ as $\ell (\xi)= \sum_i \ell_{i} (\xi)$ with components $\ell_{i} (\xi) \in H^{4i} (X;\QQ)$ for any stable topological microbundle $\xi$ over $X$. Here we use the notation as in \cite[13B]{Wall(1999)}, which is in the PL-case, the use in the topological case is again justified by \cite{Kirby-Siebenmann(1977)}.

Denote by $\hat{f} \co X \ra \G/\TOP$ the map corresponding to the degree one normal map $(f,\ol f) \co M \ra X$ under the bijection $\sN (X) \cong [X,\G/\TOP]$ of \eqref{eqn:normal-invariants-general-formula}. According to \cite[Chapter 13B, page 188]{Wall(1999)} there exists a characteristic class $\ell (\G/\TOP) \in H^{4\ast} (\G/\TOP;\QQ)$ such that 
\begin{equation} \label{eqn:char-class-formula-surgery-obstr-1-ctd-case-dim-4i}
	\theta (f, \ol f) = (\ell (X) \cdot \hat{f}^{\ast} \ell (\G/\TOP)) \; [X] \in L_{4i} (\ZZ) = \ZZ.
\end{equation}

The equations \eqref{eqn:surgery-obstruction-is-difference-of-signatures} and \eqref{eqn:char-class-formula-surgery-obstr-1-ctd-case-dim-4i} together give a relationship between the surgery obstruction and the coefficients of the $\ell (M)$ or equivalently the coefficients of $\ell (\xi)$ since $\tau_M \cong f^{\ast} \xi$. Note that $\ell (\xi)$ a-priori differs from $\ell (X)$ and their difference can be used to calculate $\theta (f,\ol f)$. In fact as explained in \cite[page 210]{Davis(2000)} in a slightly different notation we have that 
\begin{equation} \label{eqn:L-of-xi-versus-L-of-G-mod-TOP}
	\ell (\xi) = (8 \cdot \hat{f}^{\ast} \ell (\G/\TOP)+1) \cdot \ell (X),
\end{equation}
which gives $\ell (\xi)$ as a function of $\hat f$, a fact which will be used in Section \ref{sec:higher-str-set-for-cp}. We also have
\begin{equation} \label{eqn:}
	\theta (f, \ol f) = (1/8) \cdot (\ell (\xi) - \ell (X)) [X] \in \ZZ.
\end{equation} 

Coming  back to the case of $\sS (X \times D^{m},X \times S^{m-1})$ for $m \geq 3$ we note that the $\pi-\pi$-theorem of \cite[Ch 4]{Wall(1999)} and the homotopy invariance of normal invariants tell us that 
\begin{equation} \label{eqn:str-set-rel-not-rel-bdry}
	\sS (X \times D^{m},X \times S^{m-1}) \cong \sN (X \times D^{m},X \times S^{m-1}) \cong \sN (X).
\end{equation}

The almost $4$-periodicity for normal invariants and $L$-groups has as a consequence an almost $4$-periodicity for higher structure sets. This was established by Siebenmann \cite{Kirby-Siebenmann(1977)} and the precise statement is that for any compact manifold $X$ with boundary $\del X$ which might be empty we have an exact sequence of abelian groups
\begin{equation} \label{eqn:siebenmann-periodicity}
	0 \ra \sS^{s}_{\del} (X) \xra{CW} \sS^{s}_{\del} (X \times D^{4}) \xra{t} \ZZ.
\end{equation} 
The map $t$ is the zero map if $\del X \neq \emptyset$. The map $CW$ was a zig-zag of maps in the original source, but Cappell and Weinberger provided us in \cite{Cappell-Weinberger(1985)} with a geometric description, see also \cite{Crowley-Macko(2011)}. 

This leaves us with a smaller number of cases to calculate, namely those of $L \times D^m$ for $m=0,\ldots,7$. We do this, but as we will see our calculations also turn out to be independent of this periodicity result, which is perhaps also an interesting point. The periodicity is also related to the $\rho$-invariant from Section~\ref{sec:the-rho-invariant} as we discuss there.

Next we discuss some properties of the higher simple structure sets which hold specifically for manifolds we are interested in.

We start with the join construction. The main idea is that if a group $G$ (in our case $G \leq S^{1}$) acts freely on spheres $S^{a}$ and $S^{b}$ then the natural extension of this action on the join $S^{a+b+1} = S^{a} \ast S^{b}$ is also free. For the corresponding operation on fake lens spaces we will use notation $L(\alpha \ast \beta) = L (\alpha) \ast L (\beta)$. This operation gives certain maps between simple structure sets of lens spaces of different dimensions with the same fundamental group \cite{Wall(1999)}. Madsen and Rothenberg noticed that this construction can be modified to obtain also maps between higher simple structure sets. The construction appears at the end of paragraph 2 in \cite{Madsen-Rothenberg(1989)}, see also \cite{Macko(2007)}. It may be described as an iterated cone construction. Note that the join may be seen as a union of cones and this is one idea in generalizing Wall construction to the iterated construction of \cite{Madsen-Rothenberg(1989)}. If $L (\alpha_1)$ is the standard $1$-dimensional lens space we obtain in this way maps between higher simple structure sets which we call the {\it suspension maps} and we denote them
\begin{equation} \label{eqn:suspension-higher-str-sets}
	\Sigma \co \sS^{s}_{\del} (L^{2d-1} (\alpha) \times D^{m}) \ra \sS^{s}_{\del} (L^{2d+1} (\alpha \ast \alpha_1) \times D^{m}).
\end{equation}
An analogous map exists also for the complex projective spaces (that means when $G = S^{1}$).

Another piece of structure is functoriality with respect to restricting the group actions, which allows us to map between the higher simple structure sets of fake lens spaces of the same dimension but with different fundamental groups and also to map the higher simple structure sets of complex projective spaces to the higher simple structure sets of fake lens spaces. Given $H < G \leq S^1$ restricting action provides us with fiber bundles
\begin{equation} \label{eqn:transfer-from-G-to-H}
	p_H^G \co L(\alpha|_{H}) \lra L(\alpha),
\end{equation}
which induce the vertical ``transfer'' maps in the following diagram
\begin{equation} \label{eqn:transfer-str-sets-and-ni}
	\begin{split} 
		\xymatrix{
		 \sS^s_{\del} (L(\alpha) \times D^{m}) \ar[r]^{\eta} \ar[d]_{(p_H^G)^{!}} & \sN_{\del} (L(\alpha)  \times D^{m}) \ar[d]^{(p_H^G)^{!}} \\
		 \sS^s_{\del} (L(\alpha|_{H})  \times D^{m}) \ar[r]^{\eta} & \sN_{\del} (L(\alpha|_{H}) \times D^{m}).
		}
	\end{split}
\end{equation}


\section{The long surgery exact sequence for a complex projective space} \label{sec:higher-str-set-for-cp}


For our calculation we also need to deal with a version of our problem for the complex projective spaces. This case is easier due to the triviality of the fundamental group, but at the same time it will illustrate the strategy which we will employ later. We assume $d \geq 2$ and $k \geq 1$ from now on. 

The connection with lens spaces is that for $H=\ZZ/N$ and $G = S^{1}$ and $\alpha_1$ the standard action of $H$ on $S^{2d-1}$ the map $p_H^G \co L^{2d-1} (\alpha_1) \lra \CC P^{d-1}$ from \eqref{eqn:transfer-from-G-to-H} gives via Diagram~\eqref{eqn:transfer-str-sets-and-ni} maps from the higher structure sets and normal invariants of $\CC P^{d-1}$ to the higher structure sets and normal invariants of $L^{2d-1} (\alpha_1)$.

The complex projective space $\CC P^{d-1}$ can be viewed as the quotient of the diagonal $S^1$-action on $S^{2d-1} = S^1 \ast \cdots
\ast S^1$ ($d$-factors). As a real manifold it has dimension $2d-2$ and $\pi_1 (\CC P^{d-1}) = 1$. Hence from (\ref{eqn:surgery-exact-sequence}) we have that for $n-1 = 2d-2+m$ the long surgery exact sequence for $\CC P^{d-1}$ splits into the short exact sequences
\begin{equation} \label{ses-cp^d-1}
0 \ra \sS_\del (\CC P^{d-1} \times D^m) \ra \sN_{\del} (\CC P^{d-1} \times D^m) \xra{\theta}
L_{n-1}(\ZZ) \ra 0,
\end{equation}
since in the simply connected case the map $\theta$ is always surjective \cite{Madsen-Milgram(1979)}. The last term in \eqref{ses-cp^d-1} is $0$ if $n-1$ is odd, so it is convenient to split the discussion into two cases, namely when $m$ is even and odd. For the normal invariants we have from \eqref{eqn:normal-invariants-general-formula} and \eqref{eqn:htpy-grps-g-mod-top} and using the Atiyah-Hirzberuch spectral sequence that
\begin{equation}
	\begin{split} \label{eqn:normal-invariants-cp^d-1}
	\sN_\del (\CC P^{d-1} \times D^{2k+1}) \cong & \quad 0, \\
	\sN_\del (\CC P^{d-1} \times D^{2k}) \cong & \bigoplus_{i=1}^{\infty}
	H^{4i} (\CC P^{d-1}_+ \wedge S^{2k} ;\ZZ) \oplus \\ 
	& \bigoplus_{i=1}^{\infty} H^{4i-2} (\CC P^{d-1}_+ \wedge S^{2k};\ZZ/2),
	\end{split}
\end{equation}
where of course all but a finite number of summands are zero. Further we can identify the factors
\begin{align}
\bs_{4i} & \co \sN_{\del}(\CC P^{d-1} \times D^{2k}) \ra H^{4i} (\CC P^{d-1}_+ \wedge S^{2k};\ZZ) \cong \ZZ \cong L_{4i} (\ZZ) \\
\bs_{4i-2} & \co \sN_{\del}(\CC P^{d-1} \times D^{2k}) \ra H^{4i-2} (\CC P^{d-1}_+ \wedge S^{2k};\ZZ_2)
\cong \ZZ/2 \cong L_{4i-2} (\ZZ)
\end{align}
as surgery obstructions of degree one normal maps obtained from an element $(f,\overline{f}) \co M \ra \CC P^{d-1} \times D^{2k}$ of $\sN_{\del}(\CC P^{d-1} \times D^{2k})$ by first making $f$ transverse to $\CC P^{j}  \times D^{2k}$ and then taking the surgery obstruction of the degree one map obtained by restricting to the preimage of $\CC P^{j}  \times D^{2k}$. Here $j = 2i-k$ when we want $\bs_{4i}$ and $j = 2i-k-1$ when we want $\bs_{4i-2}$. The maps $\bs_{2i}$ are called the {\it splitting invariants}. This description is obtained analogously to \cite[14C]{Wall(1999)} building on \cite[13B]{Wall(1999)} which in turn builds on \cite{Sullivan(1996)}. We will sometimes use (\ref{eqn:normal-invariants-cp^d-1}) to identify the elements of $\sN_{\del} (\CC P^{d-1} \times D^{2k})$ by $s = (s_{2i})_i$.

The surgery obstruction map $\theta$ takes the top summand of $\sN_{\del} (\CC P^{d-1} \times D^{2k})$ isomorphically onto $L_{2d-2+2k} (\ZZ)$. Hence the short exact sequence (\ref{ses-cp^d-1}) splits and we obtain a bijection of $\sS_{\del} (\CC P^{d-1} \times D^{2k})$ given by the splitting invariants $\bs_{2i}$ for $k \leq i \leq k+d-2$:
\begin{equation} \label{eqn:ss-cp^d-1}
\bigoplus_{k \leq i \leq k+d-2} \bs_{2i} \co \sS_{\del} (\CC P^{d-1} \times D^{2k}) \xra{\cong} \bigoplus_{k \leq i \leq k+d-2} L_{2i} (\ZZ).
\end{equation}
Notice that if we compare $\sS (\CC P^{d-1})$ with $\sS_{\del} (\CC P^{d-1} \times D^{2k})$ we have one more summand corresponding to $\bs_{2k}$ which in case $k=2l$ corresponds to the extra $\ZZ$-summand in  \eqref{eqn:siebenmann-periodicity}.


Later we will need to identify the indexes $i$ of the splitting invariants $\bs_{4i}$ in the above isomorphisms \eqref{eqn:normal-invariants-cp^d-1} and \eqref{eqn:ss-cp^d-1}. They depend on the parity of $d$ and $k$, so to this end we introduce the following notation, where $I_{4}^{N} (d,k)$ is the indexing set of the normal invariants with indexes divisible by $4$ and $I_{4}^{S} (d,k)$ is the indexing set of the higher  structure set with indexes divisible by $4$ in both cases for $k \geq 1$. Note that the dimension of the manifolds involved is $2d-2+2k$.  

The set $I_{4}^{N} (d,k)$ is defined as the set of $i \in \ZZ$ such that

\

\begin{center}
\begin{tabular}{ |c|c|c| } 
 \hline
 $I_{4}^{N} (d,k)$ & $k = 2l$ & $k = 2l+1$ \\ 
 \hline
 $d = 2e$ & $l \leq i \leq e+l-1$ & $l+1 \leq i \leq e+l$ \\ 
 \hline 
 $d = 2e+1$ & $l \leq i \leq e+l$ & $\quad l+1 \leq i \leq e+l$ \\ 
 \hline
\end{tabular}
\end{center}

\

The set $I_{4}^{S} (d,k)$ is defined as the set of $i \in \ZZ$ such that

\

\begin{center}
\begin{tabular}{ |c|c|c| } 
 \hline
 $I_{4}^{S} (d,k)$ & $k = 2l$ & $k = 2l+1$ \\ 
 \hline
 $d = 2e$ & $l \leq i \leq e+l-1$ & $l+1 \leq i \leq e+l-1$ \\ 
 \hline 
 $d = 2e+1$ & $l \leq i \leq e+l-1$ & $\quad l+1 \leq i \leq e+l$ \\ 
 \hline
\end{tabular}
\end{center}

\ 

Hence, for example, for the free part we have
\begin{equation} \label{eqn:free-part-ss-cp^d-1}
\bigoplus_{i \in I_{4}^{S} (d,k)} \bs_{4i} \co \textup{Free} \; \sS_{\del} (\CC P^{d-1} \times D^{2k}) \xra{\cong} \bigoplus_{i \in I_{4}^{S} (d,k)} L_{4i} (\ZZ) \cong \bigoplus_{i \in I_{4}^{S} (d,k)} \ZZ.
\end{equation}

Similarly it is possible to identify the indexing sets $I^{N}_{2} (d,k)$ and $I^{S}_{2} (d,k)$ of the $\ZZ/2$-summands.

When studying the $\rho$-invariant later in Section~\ref{sec:the-rho-invariant} we will also need to understand the structure sets of the closed manifolds $\CC P^{d-1} \times S^{2k}$ to some extent. Note that we have a map 
\begin{align}
	\begin{split} \label{eqn:cp-times-disk-to-cp-times-sphere}
		\glue \co \sS_{\del} (\CC P^{d-1} \times D^{2k}) & \ra \sS (\CC P^{d-1} \times S^{2k}) \\
		[h \co Q \ra \CC P^{d-1} \times D^{2k}] & \mapsto [\glueh \co Q(h) = Q \cup_{\del h} \CC P^{d-1} \times D^{2k} \ra \CC P^{d-1} \times S^{2k}] 
	\end{split} 
\end{align}
where $\glueh$ is the obvious map. An analysis analogous to the first part of this section shows that the structure set $\sS (\CC P^{d-1} \times S^{2k})$ is isomorphic to a sum of several copies of $\ZZ$ and $\ZZ/2$ via a set of splitting invariants $\bs_{4j,0}$ and $\bs_{4i,1}$ along submanifolds
\begin{equation} \label{eqn:splitting-invariants-for-cp-times-sphere}
	\CC P^{2j} \times \{ \ast \} \; \textup{for} \; \bs_{4j,0} \quad \textup{and} \quad \CC P^{2i-k} \times S^{2k} \; \textup{for} \; \bs_{4i,1}.
\end{equation}
The map ``$\glue$'' sends $\bs_{4i}$ to $\bs_{4i,1}$ and hence the elements in its image have $\bs_{4j,0} = 0$. 

What we will need in Section~\ref{sec:the-rho-invariant} is a relationship between the splitting invariants $\bs_{4i}$ and the $\ell$-class of $Q(h)$. This will be obtained in an analogous way as by Wall in \cite[14C]{Wall(1999)} for the case of $X = \CC P^{d-1}$ but we have to make a couple of adjustments. Recall that we have
	\[
		H^{\ast} (\CC P^{d-1} \times S^{2k};\ZZ) \cong \ZZ [x,y]/(x^d,y^2) \quad |x| = 2, \; |y| = 2k.
	\]
Setting $\bar x = (\glueh)^{\ast} (x)$ and $\bar y = (\glueh)^{\ast} (y)$ we obtain the isomorphism
	\[
		H^{\ast} (Q(h);\ZZ) \cong \ZZ [\bar x, \bar y]/({\bar x}^d, {\bar y}^2 ) \quad |\bar x| = 2, \; |\bar y| = 2k.
	\]
Let
	\[
		\ell (Q(h)) = \sum_{\substack{u = 0, \ldots, d-1 \\ v = 0,1}} \alpha_{u,v} {\bar x}^{u} {\bar y}^{v}
	\]
and remember that $\alpha_{u,v} = 0$ if $u+k \cdot v$ is odd. If $(\glueh,\ol \glueh) \co Q(h) \ra \CC P^{d-1} \times S^{2k}$ is the associated degree one normal map with the bundle map $\ol \glueh \co \tau_{Q(h)} \ra \xi $ this means that 
	\[
		\ell (\xi) = \sum_{\substack{u = 0, \ldots, d-1 \\ v = 0,1}} \alpha_{u,v} {x}^{u} {y}^{v}.
	\] 
Analogously let
	\[
		\hat{\glueh}^{\ast} \ell (\G/\TOP) = \sum_{\substack{u = 0, \ldots, d-1 \\ v = 0,1}} \beta_{u,v} {x}^{u} {y}^{v}.
	\]
and remember that $\beta_{u,v} = 0$ if $u+k \cdot v$ is odd. We know from \eqref{eqn:L-of-xi-versus-L-of-G-mod-TOP} that
\[
	\ell (\xi) = (8 \cdot \hat{\glueh}^{\ast} \ell (\G/\TOP)+1) \cdot \ell (\CC P^{d-1} \times S^{2k}).
\] 
We want to show that the coefficients $\alpha_{u,v}$ are linear in $\bs_{4i}$. Hence it is enough to show that $\beta_{u,v}$ are linear in $\bs_{4i}$. To this end we recall that the map ``glue'' maps the splitting invariants $\bs_{4i}$ to $\bs_{4i,1}$ and these are surgery obstructions of the restrictions of $\glueh$ to the preimage $W_{2i}$ of $\CC P^{2i-k} \times S^{2k}$. Now using formula \eqref{eqn:char-class-formula-surgery-obstr-1-ctd-case-dim-4i} we get 
\begin{equation} \label{eqn:splitting-invariants-vs-pullback-of-l-G-TOP-1}
	8 \cdot \bs_{4i} (h) = 8 \cdot \bs_{4i,1} (\glueh) = (\ell (\CC P^{2i-k} \times S^{2k}) \cdot \hat{\glueh}^{\ast} \ell (\G/\TOP)) [\CC P^{2i-k} \times S^{2k}].
\end{equation}
The splitting invariants $\bs_{4i,0}$ are surgery obstructions of the restrictions of $\glueh$ to the preimage $W_{2i}$ of $\CC P^{2i} \times \{ \ast \}$ and using formula \eqref{eqn:char-class-formula-surgery-obstr-1-ctd-case-dim-4i} we get 
\begin{equation} \label{eqn:splitting-invariants-vs-pullback-of-l-G-TOP-0}
	8 \cdot \bs_{4i,0} (\glueh) = (\ell (\CC P^{2i} \times \{ \ast \}) \cdot \hat{\glueh}^{\ast} \ell (\G/\TOP)) [\CC P^{2i} \times \{ \ast \}].
\end{equation}

These are similar to the equations on the top half of page 203 in \cite[14C]{Wall(1999)} and we proceed analogously making necessary adjustments. Recall that $\ell (S^{2k}) = 1$ and hence we obtain that $\ell (\CC P^{2i-k} \times S^{2k}) = \ell (\CC P^{2i-k})$. Denote
\begin{equation} \label{eqn:l-of-cp-j-k}
	\ell (\CC P^{j} \times S^{2k}) = \sum_{w=0}^{\lfloor j/2 \rfloor} \gamma_{j,w} x^{2w} 
\end{equation}
for appropriate $\gamma_{j,w}$, where we know $\gamma_{j,0} = 1$.

Since $[\CC P^{2i-k} \times S^{2k}]$ is cohomologically dual to $x^{2i-k}y$ the equation \eqref{eqn:splitting-invariants-vs-pullback-of-l-G-TOP-1} gives
\begin{equation}
	8 \cdot \bs_{4i,1} (\glueh) = \sum_{u} \gamma_{2i-k,i-(u+k)/2} \beta_{u,1}.
\end{equation}
by extracting the coefficient of $x^{2i-k}y$ in the product of the cohomology classes. Varying $i$ this gives a system of linear equations and since we know $\gamma_{j,0} = 1$ it has a unique solution and we obtain that $\beta_{u,1}$ are linear in $\bs_{4i,1} (\glueh) = \bs_{4i} (h)$. 

Similarly since $[\CC P^{2i} \times \{ \ast \}]$ is cohomologically dual to $x^{2i}$ the equation \eqref{eqn:splitting-invariants-vs-pullback-of-l-G-TOP-0} gives
\begin{equation}
	8 \cdot \bs_{4i,0} (\glueh) = \sum_{u} \gamma_{2i,i-u/2} \beta_{u,0}.
\end{equation}
by extracting the coefficient of $x^{2i}$ in the product of the cohomology classes. Varying $i$ this gives again a regular system of linear equations, since we know $\gamma_{j,0} = 1$. Because $8 \cdot \bs_{4i,0} (\glueh)=0$ we obtain that $\beta_{u,0}=0$ for all $u$. 

Putting both cases together we have the following proposition. 

\begin{prop} \label{prop:alpha-and-beta-depend-linearly-on-s-4i}
	With the above notation the coefficients $\beta_{u,v}$ and hence also the coefficients $\alpha_{u,v}$ depend linearly on $\bs_{4i}$.
\end{prop}


\section{The long surgery exact sequence for a lens space} \label{sec:lens-times-disk}


Now we turn to the higher structure sets of fake lens spaces. We note that any fake lens space $L^{2d-1} (\alpha)$ is homotopy equivalent to a lens space $L^{2d-1} (\alpha_{(u,1,\cdots,1)})$ where $\alpha_{(u_1,\cdots,u_d)}$ denotes the linear action of $\ZZ/N$ on $S^{2d-1}$ given by multiplication by $e^{2 u_j \pi i}$ on the $j$-th complex coordinate of $S^{2d-1} = S (\CC^{d})$, see \cite[14E]{Wall(1999)}. Since a homotopy equivalence induces an isomorphism on higher structure sets $\sS^{s}_{\del} (X \times D^{m})$, see \cite{Ranicki(1992)}, \cite{Ranicki(2009)}, it is enough to calculate the higher simple structure sets of $L^{2d-1} (\alpha_{(u_1,1,\cdots,1)})$ for $u_1 = 1, \ldots, N-1$. For simplicity we will work with the case $u_1 = 1$ here. The other cases yield the same results, only the algebra gets a little more complicated and is solved in the same way as in \cite{Macko-Wegner(2009)}. Therefore we will abbreviate from now on $L^{2d-1} = L^{2d-1} (\alpha_{(1,\ldots,1)})$ or simply $L$ if the dimension is clear. Moreover, we assume $N = 2^K$.

We start by summarizing what we already know about the terms in~\eqref{eqn:surgery-exact-sequence} for $X = L^{2d-1}$. The $L$-theory we need is described in the following proposition from \cite{Hambleton-Taylor(2000)}. The symbol $R_{\CC} (G)$ denotes the complex representation ring of a group $G$ and the superscripts $\pm$ denote the $\pm$-eigenspaces with respect to the involution given by complex conjugation. The symbol $\Gsign$ means the $G$-signature and Arf is the Arf invariant.

\begin{thm} \cite{Hambleton-Taylor(2000)} \label{L(G)}
For $G = \ZZ/N$ with $2|N$ we have that
\begin{align*}
L^s_n (\ZZ G) & \cong
\begin{cases}
4 \cdot R_{\CC}^+ (G) & n \equiv 0 \; (\mod 4) \; (\Gsign, \; \mathrm{real}) \\
0 & n \equiv 1 \; (\mod 4) \\
4 \cdot R_{\CC}^- (G) \oplus \ZZ/2 & n \equiv 2 \; (\mod 4) \;
(\Gsign, \; \mathrm{purely}
\; \mathrm{imaginary}, \mathrm{Arf}) \\
\ZZ/2 & n \equiv 3 \; (\mod 4) \; (\mathrm{codimension} \; 1 \;
\mathrm{Arf})
\end{cases} \\
\widetilde L^s_{2k} (\ZZ G) & \cong 4 \cdot \RhG^{(-1)^k} \;
\textit{where} \; \RhG^{(-1)^k} \; \textit{is} \; R_{\CC}^{(-1)^k}
(G) \; \textit{modulo the regular representation.}
\end{align*}
\end{thm}

For the normal invariants~$\sN_{\del} (Y) \cong [Y/\del Y;\G/\TOP]$, using localization at $2$ and away from $2$, we have in general the following homotopy pullback square \cite{Madsen-Milgram(1979)}
\begin{equation} \label{eqn:htpy-type-of-g-top}
\begin{split}
	\xymatrix{
	\G/\TOP \ar[r] \ar[d] & \prod_{i > 0} K(\ZZ_{(2)},4i) \times K(\ZZ/2,4i-2) \ar[d] \\
	\BO[1/2] \ar[r] & \BO_{\QQ} \simeq \prod_{i > 0} K(\QQ,4i)
	}
\end{split}
\end{equation}
which induces a Mayer-Vietoris sequence for the homotopy sets of mapping spaces. For the products $Y = L \times D^{m}$ we notice that 
\[
Y/\del Y = L \times D^{m} / L \times S^{m-1} \simeq L_{+} \wedge S^{m} \simeq (L \wedge S^{m}) \vee S^{m}.
\]
Hence we can use the known calculations of both summands to obtain the following calculation
\begin{equation} \label{ni-lens-space-times-disk}
\sN_{\del} (L \times D^{m}) \cong \bigoplus_{i=1}^{\infty} H^{4i} (L_{+} \wedge S^{m};\ZZ) \oplus
\bigoplus_{i=1}^{\infty} H^{4i-2} (L_{+} \wedge S^{m};\ZZ/2). 
\end{equation}
At this point it is useful to split the discussion into the case when $m$ is odd and when $m$ is even.

\nin \textbf{Case $m=2k+1$.} We have
\begin{equation} \label{ni-lens-space-times-odd-disk}
	\sN_{\del} (L \times D^{2k+1}) \cong \bigoplus_{i=1}^{\infty} H^{4i} (L_{+} \wedge S^{2k+1};\ZZ) \oplus
	\bigoplus_{i=1}^{\infty} H^{4i-2} (L_{+} \wedge S^{2k+1};\ZZ/2).
\end{equation}
The first summand is zero except it is $\ZZ$ in two instances, when $(d,k)=(2e,2l)$, so that $2d-1+2k+1 = 4(e+l)$, and when $(d,k)=(2e+1,2l+1)$, so that $2d-1+2k+1=4(e+l+1)$. The other summands become $\ZZ/2$ until we reach the dimension of $2d-1+2k+1$. We denote these summands by $\bt_{4i-2}$ and the indexing set for the $i$ is denoted $J_{2}^{N} (d,k,\textup{odd})$.


\nin \textbf{Case $m=2k$.} This case is basically a shifted copy of the case $k=0$ plus a summand coming from the sphere $S^{2k}$. Similarly as in the complex projective case we introduce notation $\bt_{2i}$ for the invariants given by the respective summands, although in the present case we do not have a simple identification as splitting invariants. Nevertheless we will see that these invariants are closely related to $\bs_{2i}$. Note that we have
\begin{equation} \label{eqn:cohlgy-L-plus-smash-S}
	H^{4i} (L_+ \wedge S^{2k};\ZZ) \cong \begin{cases} \ZZ & k = 2l, i = l \\ \ZZ/{2^K} & 2k < 4i < 2(d+k)-1  \end{cases} 
\end{equation}
and we denote the summands
\begin{align}
\bt_{4i} & \co \sN_{\del} (L \times D^{2k}) \ra H^{4i} (L_+ \wedge S^{2k};\ZZ) \cong \ZZ/{2^K} \; \textup{or} \; \ZZ \\
\bt_{4i-2} & \co \sN_{\del} (L \times D^{2k}) \ra H^{4i-2} (L_+ \wedge S^{2k};\ZZ/2) \cong \ZZ/2.
\end{align}

Similarly as in the case of $\CC P^{d-1}$ it is convenient to introduce the indexing set $J^{N}_{4} (d,k)$ of those $i$ for which the invariants $\bt_{4i}$ are non-zero.

\

\begin{center}
\begin{tabular}{ |c|c|c| } 
 \hline
 $J_{4}^{N} (d,k)$ & $k = 2l$ & $k = 2l+1$ \\ 
 \hline
 $d = 2e$ & $l \leq i \leq e+l-1$ & $l+1 \leq i \leq e+l$ \\ 
 \hline 
 $d = 2e+1$ & $l \leq i \leq e+l$ & $\quad l+1 \leq i \leq e+l$ \\ 
 \hline
\end{tabular}
\end{center}

\

Similarly one could define $J^{N}_{2} (d,k)$. Though elementary, it is also helpful to similarly put into a table the dimension $n = 2d-1+2k$ of $L^{2d-1} \times D^{2k}$ in terms of parity of $d$ and $k$.

\

\begin{center}
\begin{tabular}{ |c|c|c| } 
 \hline
 $n$ & $k = 2l$ & $k = 2l+1$ \\ 
 \hline
 $d = 2e$ & $4(e+l)-1$ & $4(e+l)+1$ \\ 
 \hline 
 $d = 2e+1$ & $4(e+l)+1$ & $ 4(e+l)+3$ \\ 
 \hline
\end{tabular}
\end{center}

\

Coming back to the general case of any $m$ we obtain even more information from the following proposition. 
\begin{prop} \label{prop:theta-for-L-x-D}
\begin{enumerate}
\item If $n=2d-1+2k=4u-1$ then the map
\[
\theta \co \sN_{\del}(L^{2d-1} \times D^{2k}) \ra L^s_{4u-1}(\ZZ G) =
\ZZ/2
\]
is given by $\theta (x) = \bt_{4u-2} (x) \in  \ZZ_2$.
\item
The map
\[
\theta \co \sN_\partial(L^{2d-1} \times D^{2k+1}) \ra L^s_{2d+2k}(\ZZ G)
\]
maps onto the summand $L_{2d+2k}(\ZZ)$.
\end{enumerate}
\end{prop}

\begin{proof}
	For part (1) in the case $N=2$ we refer the reader to \cite[Section 4]{Madsen-Rothenberg(1989)}. The case $N=2^K$ is obtained by passing from $G$ to $\ZZ/2$ by restricting the action to obtain the transfer maps \eqref{eqn:transfer-str-sets-and-ni} which are surjective on normal invariants. Since the map $\theta$ is the surgery obstruction map in both cases, if it were trivial for $N = 2^K$, meaning any degree one normal map would be normally cobordant to a simple homotopy equivalence, then the transfer of the normal cobordism would be a normal cobordism for the transferred problem and hence from the surjectivity of the transfer map we would obtain the triviality of the map $\theta$ for $N = 2$ which is a contradiction. 
	
	Part (2) is the already mentioned general statement in topological surgery, due to the existence of the Milnor manifolds and Kervaire manifolds \cite{Madsen-Milgram(1979)}.
\end{proof}
In view of Proposition \ref{prop:theta-for-L-x-D} we denote
\begin{equation}
	\widetilde{\sN}_{\del} (L \times D^{m}) := \ker \theta \co \sN_{\del}(L^{2d-1} \times D^{m}) \ra L^s_{2d-1+m}(\ZZ G)
\end{equation}
and the corresponding indexing sets as $J_{4}^{tN} (d,k)$, $J_{2}^{tN} (d,k)$ and $J_{2}^{tN} (d,k,\textup{odd})$. Notice that we have $J_{4}^{tN} (d,k) = J_{4}^{N} (d,k)$.

We can now summarize what we know. Our information is enough to solve the case $m=2k+1$, the other case will take more time.

\nin \textbf{Case $m=2k+1$.}

We have the isomorphism
\begin{equation} \label{eqn:end-result-odd-disk}
	\sS^{s}_{\del} (L \times D^{2k+1}) \cong \widetilde{\sN}_{\del} (L \times D^{2k+1}) \cong \bigoplus_{J_{2}^{tN} (d,k,\textup{odd})} \ZZ/2.
\end{equation}

\nin \textbf{Case $m = 2k$.}

\nin  We obtain the short exact sequence
\begin{equation} \label{ses-lens-2d-1}
0 \ra \wL^s_{2d+2k} (\ZZ G) \xra{\partial} \sS^{s}_{\del}  (L^{2d-1} \times D^{2k})
\xra{\eta} \widetilde{\sN}_{\del}(L^{2d-1} \times D^{2k}) \ra 0,
\end{equation}
where
\begin{align}
\begin{split} \label{eqn:tilde-N-L-x-D}
n = 4u-1 \; : \; \widetilde{\sN}_{\del} (L^{2d-1} \times D^{2k}) & = \mathrm{ker} \; \big (
\bt_{4u-2} \co {\sN}_{\del} (L^{2d-1} \times D^{2k}) \ra \ZZ_2 \big ), \\
n = 4u+1 \; : \; \widetilde{\sN}_{\del} (L^{2d-1} \times D^{2k}) & = \sN_{\del} (L^{2d-1} \times D^{2k}).
\end{split}	
\end{align}
It will be convenient to use the decomposition
\begin{equation} \label{red-ni-lens-spaces}
\widetilde \sN_{\del} (L^{2d-1} \times D^{2k}) \cong T_{F} (d,k) \oplus T_{N} (d,k) \oplus T_{2} (d,k),
\end{equation}
where 
\[
T_F (d,k) \cong \begin{cases} \ZZ (t_{4l}) & k = 2l \\ 0 & k = 2l+1 \end{cases}
\]
and 
\[ 
T_N (d,k) = \bigoplus_{i \in rJ_{4}^{tN} (d,k)} \ZZ/N (t_{4i}), \quad  T_2 (d,k) = \bigoplus_{i\in J_{2}^{tN} (d,k)} \ZZ/2 (t_{4i-2}).
\]
Here $rJ_{4}^{tN} (d,k) = J_{4}^{tN} (d,k) \smallsetminus \{ l \}$ in case $k=2l$ and $rJ_{4}^{tN} (d,k) = J_{4}^{tN} (d,k)$ in case $k = 2l+1$. Then the cardinality of $rJ_{4}^{tN} (d,k)$ is equal to $c_N (d,k)$ from the statement of Theorem~\ref{thm:main-thm}. Similarly one may define $rJ_{2}^{tN} (d,k) = J_{2}^{tN} (d,k)$ in case $k=2l$ and $rJ_{2}^{tN} (d,k) = J_{2}^{tN} (d,k)  \smallsetminus \{ l+1 \}$ in case $k = 2l+1$. Then the cardinality of $rJ_{2}^{tN} (d,k)$ is equal to $c_2 (d,k)$ from the statement of Theorem~\ref{thm:main-thm}.

The first term in the sequence
(\ref{ses-lens-2d-1}) is understood by Theorem \ref{L(G)}, the third
term is understood by (\ref{red-ni-lens-spaces}). Hence we are left
with an extension problem, which we solve in Section \ref{sec:the-rho-invariant}. 

However, before that we mention some useful properties of the above calculations with respect to increasing $d$ via the suspension map \eqref{eqn:suspension-higher-str-sets} and with respect to changing the group $G$ via the transfer maps \eqref{eqn:transfer-str-sets-and-ni}

First note that the inclusion $L^{2d-1} \times D^{2k} \subset L^{2d+1} \times D^{2k}$ induces by transversality a restriction map on the normal invariants denoted by 
\begin{equation} \label{eqn:restriction-ni-dimension}
	\res \co \sN_{\del} (L^{2d+1} \times D^{2k}) \lra \sN_{\del} (L^{2d-1} \times D^{2k}).
\end{equation}
This map is related to the suspension homomorphism $\Sigma$ from \eqref{eqn:suspension-higher-str-sets} by the commutative diagram analogous to the diagram from \cite[Lemma 14A.3]{Wall(1999)}). In our situation we need to compose the horizontal maps $\eta$ with the projections onto the reduced normal invariants and to consider the map on $\widetilde{\sN}$ induced by \eqref{eqn:restriction-ni-dimension} (and denoted by the same symbol) so that we have the commutative diagram
\begin{equation} \label{susp-diagram}
\begin{split}
\xymatrix{
\sS^s_{\del} (L^{2d-1} \times D^{2k}) \ar[r]^{\eta} \ar[d]_{\Sigma} & \widetilde{\sN}_{\del} (L^{2d-1} \times D^{2k}) \\
\sS^s_{\del} (L^{2d+1} \times D^{2k}) \ar[r]^{\eta} & \widetilde{\sN}_{\del} (L^{2d+1} \times D^{2k}).
\ar[u]_{\res} }
\end{split}
\end{equation}
Clearly we have $t_{2i} = \res (t_{2i})$ and so the vertical map $\res$ from the diagram is an isomorphism when $n+2=2d+1+2k=4u-1$ and it is onto with kernel $\ZZ/N \oplus \ZZ/2$ when $n+2=2d+1+2k=4u+1$. A similar diagram exists also for the situation $\CC P^d = \CC P^{d-1} \ast \mathrm{pt}$ and the corresponding map $\res$ is always surjective in that case.

It is also useful to investigate Diagram \eqref{eqn:transfer-str-sets-and-ni} in the light of the calculations of the normal invariants. Of course, we study the case when the groups involved are $G < S^{1}$ so that we have the $S^{1}$-bundle $L^{2d-1} \ra \CC P^{d-1}$. In view of the isomorphisms \eqref{eqn:normal-invariants-cp^d-1} and \eqref{ni-lens-space-times-disk} the transfer map on normal invariants is just the induced map in cohomology, and hence on summands it is either isomorphism or the modulo $N$ reduction from $\ZZ$ to $\ZZ/N$. The indexing sets match. It will also be helpful to notice that the composition
\begin{equation} \label{eqn:fibering-ni-lens-x-disk-by-cp-x-disk}
	\sS_{\del} (\CC P^{d-1} \times D^{2k}) \xra{\eta} \sN_{\del} (\CC P^{d-1} \times D^{2k}) \xra{\textup{proj} \circ (p_{G}^{S^{1}})^{!}} \widetilde{\sN}_{\del} (L^{2d-1} \times D^{2k})
\end{equation}
is surjective for $n-1=2d-2+2k=4u+2$. This can be phrased as saying that any representative of any element in $\sS^{s}_{\del} (L^{2d-1} \times D^{2k})$ is normally cobordant to a representative of possibly another element of the same group which fibers over a fake $\CC P^{d-1} \times D^{2k}$. In case $n-1=2d-2+2k=4u$ this map is close to be surjective. Its image is the subgroup consisting of all but the top $\ZZ/N$ summand (the one which is the image of $\bs_{4u}$ with the index $u \in rJ_{4}^{N} (d,k)$). However, the corresponding summand in $\widetilde{\sN}_{\del} (L^{2d+1} \times D^{2k})$, which maps to this one under the map $\res$ from Diagram \eqref{susp-diagram}, is in the image of the map \eqref{eqn:fibering-ni-lens-x-disk-by-cp-x-disk} with $d$ replaced by $d+1$. This can be phrased by saying that in case $n-1=2d-2+2k=4u$ the suspension of any element of $\sS^{s}_{\del} (L^{2d-1} \times D^{2k})$ is normally cobordant to a representative of an element of $\sS^{s}_{\del} (L^{2d+1} \times D^{2k})$ which fibers over a fake $\CC P^{d} \times D^{2k}$. Compare to \cite[Lemma 14E.9]{Wall(1999)}


\section{The $\rho$-invariant} \label{sec:the-rho-invariant}


\subsection{Definition of the $\rho$-invariant} \label{subsec:def-of-rho-invariant}

Just as in \cite{Macko-Wegner(2009)} we will solve the extension problem by employing the~$\rho$-invariant. We first recall the definition of the $\rho$-invariant for closed manifolds as used in \cite[subsection 4.1]{Macko-Wegner(2009)} and its generalization from \cite{Madsen-Rothenberg(1989)}.

Let $\RhG :=  R(G) / \langle \textup{reg} \rangle$ where quotienting by $\langle
\textup{reg} \rangle$ means dividing by the regular representation and $\QQ \RhG = \QQ \otimes \RhG$.

\begin{defn}{\cite[Remark after Corollary 7.5]{Atiyah-Singer-III(1968)}} \label{defn-rho-1}
Let $X^{2u-1}$ be a closed oriented manifold with with a reference map $\lambda (X) \co X \ra BG$. Define
\begin{equation}
\rho (X,\lambda(X)) = \frac{1}{r} \cdot \Gsign(\widetilde Y) \in \QQ \RhG^{(-1)^u}
\end{equation}
for some $r \in \NN$ and $(Y,\partial Y)$ such that there exists $\lambda_Y \co Y \ra BG$ and $\partial Y = r \cdot X$. 
\end{defn}

As explained in \cite{Atiyah-Singer-III(1968)} this is well defined. There it is also explained that there is also another definition which
works for actions of compact Lie groups, in particular for $S^1$-actions, on certain odd-dimensional manifolds. Whenever the two definitions apply, they coincide. For $G < S^1$ we identify $R(G)$ with $\ZZ \Gh$ and we adopt the notation $\RhG = \ZZ [\chi] / \langle 1 +
\chi + \cdots + \chi^{N-1} \rangle$ following \cite[section 4.1]{Macko-Wegner(2009)}. As also explained in \cite[Definition 4.2]{Macko-Wegner(2009)}, or \cite[Definition 2.5]{Crowley-Macko(2011)}, the~$\rho$-invariant defines a function on simple structure sets:

\begin{defn} \label{defn:reduced-rho}
Let $X$ be a closed oriented manifold of dimension $n = (2u-1)$ with a reference map $\lambda (X) \co X \ra BG$. Define the function
\[
\wrho \co \sS^{s} (X) \ra \QQ \RhG^{(-1)^u} \quad \textup{by} \quad
\wrho ([h]) = \rho (M,\lambda (X) \circ h) - \rho (X,\lambda (X)),
\]
where the orientation on $M$ is chosen so that the homotopy
equivalence $h \co M \ra X$ is a map of degree $1$.
\end{defn}

The definition in the relative setting comes from \cite[Section 3]{Madsen-Rothenberg(1989)}. We need a little preparation.
Consider a closed oriented $a$-dimensional manifold $X$, $m \geq 1$ such that $n = a+m = 2u-1$ and an element $[h \co M \ra X \times D^m]$ in $\sS^{s}_\partial (X \times D^m)$. Let $M(h)$ be a closed manifold given by
\begin{equation} \label{defn:M(h)}
M (h) := M \cup_{\partial h} (X \times D^m).
\end{equation}
If $h$ is the identity we obtain $M (\id) \cong X \times S^m$, in general the map $h$ induces $M(h) \simeq X \times S^m$. We equip $M$ with an orientation so that $h$ is a map of degree $1$. The orientation on the closed manifold $M(h)$ can then be chosen so that it agrees with the given orientation on $M$ and it reverses the orientation on $X \times D^m$. If $X$ possesses a reference map $\lambda (X) \co X \ra BG$ then we obtain a reference map $\lambda (M(h)) \co M(h) \simeq X \times S^m \ra X \ra BG$. As explained in~\cite[Definition 2.5]{Crowley-Macko(2011)} (also a minor modification of \cite[(3.7)]{Madsen-Rothenberg(1989)}) we can make the following definition:

\begin{defn} \label{defn:reduced-rho-del}
Let $X$ be a closed oriented $a$-dimensional manifold with a reference map $\lambda (X) \co X \ra BG$ and let $m \geq 1$ be such that $n=a+m=2u-1$. Define the function
\[
\wrho_\partial \co \sS^{s}_\partial (X \times D^m) \ra \QQ \RhG^{(-1)^u}
\quad \textup{by} \quad  \wrho_\partial ([h]) := \rho (M(h),\lambda(M(h))).
\]
\end{defn}
Again this is well-defined and notice that $\wrho_\partial ([\id]) = 0$.


\subsection{Properties of the $\rho$-invariant} \label{subsec:prop-of-rho-inv}


First a basic example. For the standard $1$-dimensional lens space $L=L^1(\alpha_1)$ we have \cite[Proof of Theorem 14C.4]{Wall(1999)}
\begin{equation} \label{rho-alpha-k}
\rho (L^1(\alpha_1)) = f = \frac{1+\chi}{1-\chi} \in \QQ \RhGm.
\end{equation}
As explained in \cite[Chapter 14E, page 215]{Wall(1999)} we have
\begin{equation}
	(1-\chi)^{-1} = - (1/N) \cdot (1 + 2 \chi + 3 \chi^2 + \cdots N \chi^{N-1})
\end{equation}
which together with \eqref{rho-alpha-k} gives an expression used in calculations in \cite{Macko-Wegner(2009)}.

Next we note that the $\rho$-invariant in Definitions \ref{defn:reduced-rho} and \ref{defn:reduced-rho-del} is a homomorphism. This was shown in general in \cite{Crowley-Macko(2011)}. Also the $\rho$-invariant commutes with Siebenmann $4$-periodicity alias Cappell-Weinberger map from \eqref{eqn:siebenmann-periodicity}, see also \cite{Crowley-Macko(2011)}. As mentioned in the introduction this reduces the number of cases we need to study, but our treatment turns out to be independent of this observation. 

Recall that for the join $L (\alpha \ast \beta)$ of fake lens spaces $L (\alpha)$ and $L (\beta)$ we have \cite[chapter 14A]{Wall(1999)}
\begin{equation} \label{rho-join}
\rho (L(\alpha \ast \beta)) = \rho (L(\alpha)) \cdot \rho (L(\beta)).
\end{equation}
Recall that when $\beta$ is the standard $1$-dimensional representation this operation produces a map between the structure sets of lens spaces of different dimensions. We are interested in its generalization given by the suspension map $\Sigma$ of \eqref{eqn:suspension-higher-str-sets} and the behavior of $\wrho_{\del}$ with respect to this map. Madsen and Rothenberg show in \cite[Lemma 3.9]{Madsen-Rothenberg(1989)} that the $\wrho_{\del}$-invariant is indeed multiplicative with respect to this $\Sigma$. They use a slightly different definition of the $\wrho_{\del}$-invariant, which differs from our in a factor corresponding to the element $f$ from \eqref{rho-alpha-k}. Hence the commutativity of the diagram from~\cite[Lemma 3.9]{Madsen-Rothenberg(1989)} translates into the formula
\begin{equation} \label{eqn:rho-mult-wrt-join-higher-str-sets}
	\wrho_{\del} (\Sigma ([h])) = f \cdot \wrho_{\del} ([h]) \in \QQ \RhG^{(-1)^{u+1}}
\end{equation}
for $h \co M \ra L^{2d-1} \times D^{2k}$ representing an element in $\sS^{s}_{\del} (L^{2d-1} \times D^{2k})$

The $\rho$-invariant also behaves naturally with respect to the passage to a subgroup. How it works in our notation is explained in detail in Remark 4.4 in \cite{Macko-Wegner(2011)} if needed.











\subsection{The main diagram}
\label{sec:com-ladder}


We now have all the ingredients we need to analyze the surgery exact sequence for $X = L^{2d-1} \times D^{2k}$. Denoting $n = 2d-1+2k$ these can be summarized in the commutative ladder:
\[
\xymatrix{
0 \ar[r] & \wL_{n+1} (\ZZ G) \ar[r] \ar[d]_{G-\textup{sign}} & \sS^{s}_{\del} (L \times D^{2k}) \ar[d]^{\wrho_{\del}} \ar[r] & \wsN_{\del} (L \times D^{2k}) \ar[d]^{[\wrho_{\del}]} \ar[r] & 0 \\
0 \ar[r] & 4 \cdot \RhG^\pm \ar[r] & \QQ \RhG^{\pm} \ar[r] & \QQ \RhG^{\pm}/4 \cdot \RhG^{\pm} \ar[r] & 0  
} 
\]
By the following theorem we need to understand the kernel of $[\wrho_\del]$. 

\begin{thm} \label{thm:how-to-split-alpha-k}
Let $\bar T := \ker \big( [\wrho_{\del}]: \wsN_{\del} (L \times D^{2k}) \lra \QQ\RhG^{(-1)^{d+k}}/4 \cdot \RhG^{(-1)^{d+k}}
\big)$. Then we have
\[
\sS^{s}_{\del} (L \times D^{2k}) \cong F^{(-1)^{d+k}} \oplus \bar T
\]
where $F^{(-1)^{d+k}} := \wrho_{\del} (\sS^{s}_{\del} (L \times D^{2k}))$ is a free abelian group of rank $N/2-1$ if $d+k$ is odd and of rank $N/2$ if $d+k$ is even. 
\end{thm}

\begin{proof}
	The proof is the same as the proof of Theorem 5.1 in~\cite{Macko-Wegner(2009)}. The dimensions of $F^{\pm}$ are determined in \cite{Macko-Wegner(2009)} at the end of Section 4.1.
\end{proof}


\subsection{The generalized formula of Wall}
\label{subsec:generalized-formula-of-Wall}


In addition to the information we obtained in the previous subsection we need formulas to calculate the homomorphism $[\wrho_{\del}]$ in some cases. The formulas and the proofs are generalizations of the similar formulas in \cite[section 4.2]{Macko-Wegner(2009)}. The starting point is the following generalization of the formula of Wall from {\cite[Theorem 14C.4]{Wall(1999)}}.





\begin{thm} \label{thm:rho-formula-cp-times-disk}
Let $h \co Q \ra \CC P^{d-1} \times D^{2k}$ represent an element in the higher structure set $\sS_{\del} (\CC P^{d-1} \times D^{2k})$. Then for $t \in S^1$ we have
\[
\wrho_{S^1,\del} (t,[h]) = \sum_{i \in I_{4}^{S} (d,k)} 8 \cdot \bs_{4i} (\eta ([h])) \cdot (f^{d+k-2i} - f^{d+k-2i-2}) \in \CC 
\]
where $f = (1+t)/(1-t)$ and the indexing set $I_{4}^{S} (d,k)$ is defined in Section~\ref{sec:higher-str-set-for-cp}. 
\end{thm}

\begin{proof}
	The proof follows the same general pattern as the proof of Theorem 14C.4 in~\cite{Wall(1999)}. We have to make sure that all steps either carry over from the case of $\CC P^{d-1}$ to our case of $\CC P^{d-1} \times D^{2k}$ or we have to modify them accordingly. In particular, we have to take care of the extra summand coming from the splitting invariant $\bs_{2k}$ along $\{ \textup{pt} \} \times D^{2k}$.
	
	Firstly, starting with the homotopy equivalence $h \co Q \ra \CC P^{d-1} \times D^{2k}$ which is a homeomorphism on the boundary, we use the homeomorphism $\del h$ to form the closed manifold $Q(h) = Q \cup_{\del h} (\CC P^{d-1} \times D^{2k})$ and it is the $\rho$-invariant of $Q(h)$ that we are studying. The map $h$ also induces a homotopy equivalence $\glueh \co Q (h) \ra \CC P^{d-1} \times S^{2k}$, which induces an isomorphism on cohomology. The passage from $h$ to $\glueh$ was already explained in the second half of Section~\ref{sec:higher-str-set-for-cp} together with the behavior of the respective invariants and now we will use that knowledge. 
	
	
	The manifold $Q (h)$ can be viewed as obtained from a free action of $S^{1}$ on the product $S^{2d-1} \times S^{2k}$ and hence the Atiyah-Singer-Bott formula (7.9) from \cite{Atiyah-Singer-III(1968)} can be used. It expresses the $\rho$-invariant in terms of the $\sL$-class of $Q (h)$ as
	\[
	\rho (Q(h),t) = \varepsilon - 2^{d+k-1} \Bigg( \frac{te^{\bar x}+1}{te^{\bar x}-1}\Bigg) \sL (Q(h)) [Q(h)],
	\] 
where $\varepsilon$ is zero in our case, $t \in S^{1} \subset \CC$ and $\bar x \in H^{2} (Q(h);\ZZ)$ is the first Chern class of the associated complex line bundle and hence we can assume it is the generator $\bar x$ which we chose earlier arbitrarily. The class $\sL (X)$ is the characteristic class given by (6.5) in \cite{Atiyah-Singer-III(1968)} and is related to $\ell (X)$ for an $(2u-2)$-dimensional manifold $X$ via the equation $\ell (X) [X] = 2^{u-1} \cdot \sL (X)[X]$.
	

	
	In Proposition \ref{prop:alpha-and-beta-depend-linearly-on-s-4i} at the end of Section \ref{sec:higher-str-set-for-cp} we showed that the coefficients of $\ell (Q(h))$ are linear in the splitting invariants $\bs_{4i}$.	Hence, just as in the proof of Theorem 14C.4 in \cite{Wall(1999)}, we obtain that the $\rho$-invariant can be expressed as a linear combination of the splitting invariants. But here we have to take as the indexing set $I_{4}^{S}(d,k)$ which has one more element than $I_{4}^{S}(d,0)$ and this has to be taken into the account. So we have
	\[
		\rho (h) = a_{d+k} + b_{d+k}^j \cdot \bs_{4j} + b_{d+k}^{j+1} \cdot \bs_{4(j+1)} + \cdots + b_{d+k}^{r} \cdot \bs_{4r}, \quad \textup{where} \; I_{4}^{S}(d,k) = \{ j, \ldots, r \}.
	\]

	We proceed by induction on $d$. At the induction beginning, when $d=1$, in \cite[Theorem 14C.4]{Wall(1999)} one had to deal with the point being the quotient of the free action of $S^1$ on itself. Here, we instead have the product of the point with the sphere $S^{2k}$ seen as the quotient of the free action of $S^1$ on $S^1 \times S^{2k}$ that is trivial on the second coordinate. Since this is the boundary of $S^1 \times D^{2k+1}$ we have that  
	\[
	\rho (h) = 0
	\]
	and hence $a_{d+k} = 0$. 
	
	In the induction step we use that the $\rho$-invariant is multiplicative with respect to modified join. From \eqref{eqn:rho-mult-wrt-join-higher-str-sets} it follows that 
	\[
		b_{d+k}^{s} = f^{d+k-2s-2} b_{2s+2}^{s}
	\] 
just as in \cite[Theorem 14C.4]{Wall(1999)}. Hence we assume we have all the $b_{d+k}^{s}$ except $b_{2r+2}^{r}$, which we need to calculate. 
	
To calculate $b_{2r+2}^{r}$ we need to look at the relation between $\bs_{4r}$ and the coefficients in $\sL (Q(h))$ even closer. Let $\delta_r x^{2r-k}y$ be the coefficient of $\bs_{4r}$ in $\sL (Q(h))$. 


As in \cite[Theorem 14C.4]{Wall(1999)} consider the preimage $W_{2r}$ of $\CC P^{2r-k} \times S^{2k}$, which is the dual of $x^{2r-k}y$, with normal bundle $\nu_{2r}$ of $\iota_{r} \co W_{2r} \hookrightarrow Q(h)$ and observe that the leading $0$-dimensional term in $\sL (\nu_{2r})$ is $1$. We look at the equation 
	\[
	8 \cdot \bs_{4r} (h) = \sign (W_{2r}) = 2^{2r} \sL (W_{2r}) [W_{2r}] = 2^{2r} (\iota_{r})^{\ast} \sL (Q(h)) \cdot \sL (\nu_{2r})^{-1} [W_{2r}].
	\]
	Both sides are linear functions of $\bs_{4r}$ and the coefficient on the right hand side is $2^{2r} \cdot \delta_r$ so we get
	\[
	\delta_{r} = 2^{3-2r}.
	\]
	Finally, exactly as in \cite[Theorem 14C.4]{Wall(1999)} we obtain
	\[
	b_{2r+2}^{r} = 2^{3-2r} \cdot -2^{2r+1} \cdot 2t/(1-t)^{2} = 8 (f^2-1).
	\]
 	This finishes the proof.
\end{proof}

Next, similarly as in \cite{Macko-Wegner(2009)} we pass from the formula for the $\rho$-invariant of elements in $\sS^{s}_{\del} (\CC P^{d-1} \times D^{2k})$ to the $\rho$-invariant of elements in $\sS^{s}_{\del} (L^{2d-1} \times D^{2k})$. The idea is the same, we use the transfer maps induced by $p \co L^{2d-1} \ra \CC P^{d-1}$ followed by projection on the reduced normal invariants (see discussion at the end of Section \ref{sec:lens-times-disk}) :
\begin{equation}
	\begin{split}
		\xymatrix{
			\sS^{s}_{\del} (\CC P^{d-1} \times D^{2k}) \ar[d] \ar[r] & \sN_{\del} (\CC P^{d-1} \times D^{2k}) \ar[d] \\
			\sS^{s}_{\del} (L^{2d-1} \times D^{2k}) \ar[r] & \widetilde{\sN}_{\del} (L^{2d-1} \times D^{2k}).
		}
	\end{split}
\end{equation}
Let us now remember that we are only interested in the case $N = 2^K$ for $K \geq 1$ and hence assume this from now on. The composition from the upper left to the lower right corner is surjective if $n-1= 2(d-1)+2k = 4u+2$, which immediately gives the desired formula for $[\widetilde{\rho}_{\del}]$, see the first part of Proposition~\ref{rho-formula-lens-sp} below. If $n-1=4u$, then this composition is onto the subgroup obtained by leaving out the last $\ZZ/N$-summand. To include the last summand into the formula the same trick as in  \cite[Lemma 4.10, 4.11]{Macko-Wegner(2009)} can be used in our situation. Its proof only used the suspension map $\Sigma$ which is available in our case as well, the multiplicativity of the $\rho$-invariant with respect to $\Sigma$ and the rest of the proof was completely algebraic and hence carries over word-for-word to our case so that we obtain the following lemma.

\begin{lem} \label{all-summands-n=4m+2}
Let $d = 2e$ and $k=2l$ or let $d = 2e+1$ and $k = 2l+1$ so that in either case we have $n-1=2(d-1)+2k=4u$. 

Let $c \in \sS^s_{\del} (L^{2d-1} \times D^{2k})$ be such that $c = b + a$ where $b = p^{!} (\widetilde b)$ for some
$\widetilde b \in \sS_{\del} (\CC P^{d-1} \times D^{2k})$ with $s_{2i} = \bs_{2i} (\eta \widetilde b)$ and $a$ is such that there exists an $\widetilde a \in \sN_{\del} (\CC P^{d-1} \times D^{2k})$ such that $\bs_{2i} (\widetilde a) = 0$ for $i < 2u$,  $s_{4u} = \bs_{4u} (\widetilde a)$ and $\eta (a) = p^{!} (\widetilde a)$. Then
\[
\widetilde\rho_{\del} (c) = \sum_{i \in I^{S}_{4} (d,k)} \!\! 8 \cdot s_{4i} \cdot (f^{d+k-2i} - f^{d+k-2i-2}) + 8 \cdot s_{4u}
\cdot f + z \quad \in \quad \QQ\RhGm
\]
for some $z \in 4 \cdot \RhGm$.
\end{lem}

Using discussion around \eqref{eqn:fibering-ni-lens-x-disk-by-cp-x-disk} this immediately gives an analogue of Proposition 4.12 from \cite{Macko-Wegner(2009)} which reads as follows.

\begin{prop} \label{rho-formula-lens-sp}
Let $t = (t_{2i})_i \in \widetilde \sN_{\del} (L^{2d-1} \times D^{2k})$. If $n-1= 2(d-1) + 2k$, then we have for the homomorphism $[\wrho_\del] \co \widetilde \sN_{\del} (L^{2d-1} \times D^{2k}) \lra \QQ \RhG^{(-1)^{d+k}} / 4 \cdot \RhG^{(-1)^{d+k}}$ that
\begin{align*}
n-1 = 4u+2 \; : \; [\widetilde \rho_{\del}] (t) & = \sum_{i \in J^{N}_{4} (d,k)} 8 \cdot
 t_{4i} \cdot f^{d+k-2i-2} \cdot (f^2-1) \\
n-1 = 4u \; : \; [\widetilde \rho_{\del}] (t) & = \sum_{i \in J^{N}_{4} (d,k) \smallsetminus \{u\}} 8 \cdot
 t_{4i} \cdot f^{d+k-2i-2} \cdot (f^2-1) + 8 \cdot t_{4u}
\cdot f.
\end{align*}
\end{prop}


\section{Calculations} \label{sec:calculations}


We want to complete the proof of Theorem \ref{thm:main-thm} by calculating the summand $\bar T = \ker [\wrho_{\del}]$ from Theorem~\ref{thm:how-to-split-alpha-k}. We achieve this in Propositions \ref{T-2} and \ref{T-N} using formulas from Proposition~\ref{rho-formula-lens-sp}. 




In \cite{Macko-Wegner(2009)} in a similar situation the reduced normal invariants were an abelian group which was a sum of $N$-torsion and $2$-torsion. As described in \eqref{red-ni-lens-spaces} here we have an extra summand from $D^{2k}$ and also the indexing is a little different. Nevertheless it is possible to make only minor changes in the setup from \cite{Macko-Wegner(2009)} to obtain our results. 

\begin{prop} \label{T-2}
We have
\[
T_2 (d,k) \subseteq \bar T.
\]
\end{prop}

\begin{proof}
By Proposition \ref{rho-formula-lens-sp} the formula for $[\wrho_{\del}]$ only depends on $t_{4i}$.
\end{proof}

Next we investigate the behavior of the map $[\wrho_{\del}]$ on the summands $T_F (d,k)$ and $T_N (d,k)$ via the formulas from Proposition \ref{rho-formula-lens-sp}. Denote 
\begin{equation} \label{eqn:bar-T-F-plus-N}
	\bar T_{F \oplus N} (d,k) = T_{F \oplus N} (d,k) \cap \bar T, \quad \textup{where} \quad T_{F \oplus N} (d,k) = T_{F} (d,k) \oplus T_{N} (d,k).
\end{equation}
The formulas are very similar to those from Proposition 4.15 from \cite{Macko-Wegner(2009)} except the source of the map has an extra copy of $\ZZ$ and the indexing is a little different. Nevertheless the same general strategy as in \cite{Macko-Wegner(2009)} can be employed to find the kernel of $[\wrho_{\del}]$. The strategy is explained at the beginning of Section 5 in  \cite{Macko-Wegner(2009)}. Briefly, first a passage from $T_N (d)$ is made to the underlying abelian group $\ZZ/N [x](d)$ of the appropriately truncated polynomial ring in the variable $x$ over $\ZZ/N$ via isomorphisms (5.1) and (5.3), so that $[\wrho]$ is transformed into maps (5.2) and (5.4); here the references are to formulas in \cite{Macko-Wegner(2009)}. Then another passage is made to the underlying abelian subgroup $\ZZ[x](d)$ the appropriately truncated polynomial ring in the variable $x$ over the integer coefficients $\ZZ$. The problem is thus translated to finding 
\begin{equation} \label{eqn:rho-hat}
\textup{the preimage} \; A_K (d) \; \textup{of} \; 4 \cdot \RhG^{(-1)^d} \; \textup{under the map} \; [\widehat{\rho}] \co \ZZ[x](d) \ra \QQ \RhG^{(-1)^d}	
\end{equation}
given by the same formulas (5.2) and (5.4) keeping in mind that the coefficients are now integers. In symbols, the situation in \cite{Macko-Wegner(2009)} was like this:
\begin{equation} \label{eqn:scheme-T-vs-Z-x}
	\xymatrix{
		T_N (d) \ar[r]^(0.4){\cong} & \ZZ/N [x](d) & \ZZ[x](d) \ar[l]_(0.4){\red_N}
	}
\end{equation}
so that the following transformations lead to the map $[\widehat{\rho}]$ given for $d=2e$ as
\begin{equation} \label{eqn:rho-hat-formula}
	t = (t_{4i}) \mapsto q_t (x) = \sum t_{4(i+1)} \cdot x^{c-i-1} \; \leadsto \; [\widehat{\rho}] (q) = 8 \cdot (f^{2}-1) \cdot q (f^{2}).	
\end{equation}
and for $d = 2e+1$ as
\begin{equation} \label{eqn:rho-hat-formula-second}
	t = (t_{4i}) \mapsto q_t (x) = \sum (t_{4(i+1)} - t_{4i}) \cdot x^{c-i-1} + t_{4} \cdot x^{c-1} \; \leadsto \; [\widehat{\rho}] (q) = 8 \cdot f \cdot q (f^{2}).	
\end{equation}
In both cases $c = \lfloor (d-1)/2 \rfloor$. Then $\bar T_N (d) = \red_{N} (A_K (d))$, and the subgroup $A_K (d)$ is shown to be equal to certain subgroup $B_K (d)$ in Theorem 5.4 of \cite{Macko-Wegner(2009)}, which is the main calculation of that paper.

In our case these maps have to be modified as follows. The indexing set is $J_{4}^{tN} (d,k)$ and for $c (a) =\lfloor (a-1)/2 \rfloor$ we use the following  truncated polynomial ring $\ZZ [x] (a) = \{ q \in \ZZ[x] \; | \: \deg(q) \leq c(a)-1 \}$. The coefficients of the polynomial $q$ are indexed as $q = q_0 + q_1 \cdot x + \cdots + q_{c(a)-1} \cdot x^{c(a)-1}$. Note that $c(2b+1)=c(2b+2)$ hence in picking $a$ below we do have some freedom which we will use in order to get compatibility between the truncation and the signs of the eigenspaces.

We have four cases in each of which the maps $\red_{N,d,k} \co q \mapsto t$ correspond to the composition of $\red_N$ and the inverse of the isomorphism $T_N (d) \xra{\cong}\ZZ/N [x](d)$ from Diagram~\eqref{eqn:scheme-T-vs-Z-x}: 

\noindent \textbf{Case $d=2e$, $k=2l$.} Hence $n=2d-1+2k=4(e+l)-1$ and we have 
\begin{equation} \label{eqn:red-2e-2l}
		\red_{N,d,k} \co \ZZ [x] (d+2) \ra T_{F \oplus N} (d,k) = T_F (d,k) \oplus T_N (d,k) 
\end{equation}
given by
\begin{equation} \label{eqn:rho-hat-del-formula-2e-2l}
		q \mapsto t=(t_{4j} := q_{(e+l-1)-j})_j \; \leadsto \; [\widehat{\rho}] (q) = 8 \cdot (f^{2}-1) \cdot q (f^{2}).
\end{equation}

\noindent \textbf{Case $d=2e+1$, $k=2l+1$.} Hence $n=2d-1+2k=4(e+l)+3$ and we have 
\begin{equation} \label{eqn:red-2e+1-2l+1}
		\red_{N,d,k} \co \ZZ [x] (d+1) \ra T_{F \oplus N} (d,k) = T_F (d,k) \oplus T_N (d,k) 
\end{equation}
given by
\begin{equation} \label{eqn:rho-hat-del-formula-2e+1-2l+1}
		q \mapsto t=(t_{4j} := q_{(e+l)-j})_j \; \leadsto \; [\widehat{\rho}] (q) = 8 \cdot (f^{2}-1) \cdot q (f^{2}).
\end{equation}

\noindent \textbf{Case $d=2e$, $k=2l+1$.} Hence $n=2d-1+2k=4(e+l)+1$ and we have 
\begin{equation} \label{eqn:red-2e-2l+1}
		\red_{N,d,k} \co \ZZ [x] (d+1) \ra T_{F \oplus N} (d,k) = T_F (d,k) \oplus T_N (d,k) 
\end{equation}
given by
\begin{equation} \label{eqn:rho-hat-del-formula-2e-2l+1}
	q \mapsto t=(t_{4j} := \sum_{v=1}^{j-l} q_{e-v})_j \; \leadsto \; [\widehat{\rho}] (q) = 8 \cdot f \cdot q (f^{2}).	
\end{equation}

\noindent \textbf{Case $d=2e+1$, $k=2l$.} Hence $n=2d-1+2k=4(e+l)+1$ and we have 
\begin{equation} \label{eqn:red-2e+1-2l}
		\red_{N,d,k} \co \ZZ [x] (d+2) \ra T_{F \oplus N} (d,k) = T_F (d,k) \oplus T_N (d,k) 
\end{equation}
given by
\begin{equation} \label{eqn:rho-hat-del-2e+1-2l}
	q \mapsto t=(t_{4j} := \sum_{v=0}^{j-l} q_{e-v})_j \; \leadsto \; [\widehat{\rho}] (q) = 8 \cdot f \cdot q (f^{2}).	
\end{equation}
In each case the truncation $\ZZ[x](a)$ is chosen so that $c(a)$ is the cardinality of the indexing set $J_{4}^{N}(d,k)$ and so that $a$ and $(d+k)$ have the same parity.


Hence we now have the maps
\begin{equation} \label{eqn:rho-hat-k-not-0}
	[\widehat{\rho}] \co \ZZ[x](a) \ra \QQ \RhG^{(-1)^{a}}
\end{equation}
with appropriate $a$ in the respective cases and we need to find the preimage $A_K (a)$ of $4 \cdot \RhG^{(-1)^{a}}$.

This means that the setup is the same as in \cite{Macko-Wegner(2009)}. Indeed, let us in particular look at the definition of the sets $B_K (a)$ on pages 17-18 of \cite{Macko-Wegner(2009)}. First it is mentioned that for each $n \in \NN$ there are universal polynomials $r_n^-$ and $r_n^+$ of degree $n$ with leading coefficient $1$ with certain properties (they are explicitly constructed later in Definition 5.26 of \cite{Macko-Wegner(2009)}). Then a definition is made
\begin{equation} \label{eqn:B_K-a}
	B_{K} (a) := \left \{ \sum_{n = 0}^{c(a)-1} a_n \cdot 2^{\textup{max}\{K-2n-2,0\}} \cdot r^{(-1)^{a}}_{n} \; | \; a_n \in \ZZ \right \}
\end{equation}
By universality we mean that the polynomials $r_n^-$ or $r_n^+$ do not depend on $a$, the same polynomial $r_n^-$ or $r_n^+$ is used in all sets $B_K (a)$. The main calculational result of \cite{Macko-Wegner(2009)} is Theorem 5.4 where it is proved that 
\begin{equation} \label{eqn:A_K-a=B_K-a}
	A_K (a) = B_K (a).
\end{equation}

\begin{prop} \label{T-N}
When $k=2l$ we have
\[
\bar T_{F \oplus N} (d,k) \cong \ZZ \oplus \bigoplus_{i =1}^{c_{N} (d,k)} \ZZ/{2^{\min\{K,2i\}}}.
\]
When $k=2l+1$ we have
\[
\bar T_{F \oplus N} (d,k) \cong \bigoplus_{i =1}^{c_{N} (d,k)} \ZZ/{2^{\min\{K,2i\}}}.
\]
\end{prop}

\begin{proof}
Equation \eqref{eqn:A_K-a=B_K-a} and the definition of $B_K(a)$ show that $A_K^k(a)$ is a free abelian subgroup of $\ZZ[x](a)$ with a basis given by polynomials $2^{\max\{K-2n-2,0\}} \cdot r_n^\pm$. Under the homomorphism $\ZZ[x](a) \ra T_{F \oplus N} (d,k)$ the subgroup $A_K (a)$ for appropriate $a$ is mapped onto a subgroup isomorphic to a direct sum as claimed by the theorem, since $c_N (d,k)$ is the cardinality of the indexing set $rJ_{4}^{N} (d,k)$.
\end{proof}

\begin{proof}[Proof of Theorem \ref{thm:main-thm}]
The proof for $L \times D^{2k}$ follows from Theorem \ref{thm:how-to-split-alpha-k}, Proposition \ref{T-N} and suitable reindexing in Proposition \ref{T-2}. The part about the invariant $\wrho_{\del}$ follows directly from Theorem \ref{thm:how-to-split-alpha-k}. The invariants $\bbr_{0}$ are taken out of the sums $\bar T_{F \oplus N} (d,k)$ of Proposition \ref{T-N} and $T_{2} (d,k)$ of Proposition \ref{T-2}, respectively. They are given as splitting invariants along $\{ \ast \} \times S^{2k}$ and hence play a special role and by extracting them from the sum we stress this point. The invariant $\bbr$ is then defined abstractly as coming from the torsion summand of Proposition \ref{T-N} and the invariant $\br$ defined abstractly as coming from the torsion summand of Proposition \ref{T-2}. The proof for $L \times D^{2k+1}$ follows from \eqref{eqn:end-result-odd-disk} by calculating the cardinality of the indexing set.
\end{proof}

\begin{proof}[Proof of Corollary \ref{cor:higehr-str-sets-L-times-S}]
	Follows from the isomorphism \eqref{eqn:higher-str-set-X-times-S} and the calculation of the structure sets $\sS^{s} (X \times D^{m},X \times S^{m-1}) = \sN (X)$ from \eqref{eqn:str-set-rel-not-rel-bdry}. The normal invariants $\sN (L)$ are calculated in (3.9) of \cite{Macko-Wegner(2009)} which gives the indexing in the statement of the corollary and the invariants $\br'$ and $\br''$ are given by the corresponding components.
\end{proof}

\section{Final Remarks} \label{sec:final-remarks}

One obvious future direction would be to try to obtain a better geometric description of the invariants $\bbr$ from Theorem \ref{thm:main-thm}.
In \cite{Macko-Wegner(2011)} there was offered an inductive obstruction theoretic description of the corresponding invariants when $m=0$. Such a description works also in the present case with the proof very similar to the case $m=0$, so we refrain from repeating it here and we refer the reader to \cite{Macko-Wegner(2011)}. Of course, even better would be a non-inductive description, but this remains open even in the case $m=0$.

Another future direction would be to address the case $N = 2^K \cdot M$ for $M$ odd, again similarly as is done in \cite{Macko-Wegner(2011)}. We plan to do this in a future work. The reason it is not included here is that the localization which is the main tool in \cite{Macko-Wegner(2011)} becomes more difficult and hence we want to deal with it separately. 

\section*{Acknowledgments}

We thank James F. Davis for a discussion about formula \eqref{eqn:L-of-xi-versus-L-of-G-mod-TOP}.

\small
\bibliography{lens-spaces}
\bibliographystyle{alpha}

\end{document}